\documentclass[12pt,twoside,reqno]{amsart}
\usepackage[colorlinks=true,citecolor=blue]{hyperref}
\usepackage{mathptmx, amsmath, amssymb, amsfonts, amsthm, mathptmx, enumerate, color}
\usepackage{graphicx}
\usepackage{epstopdf}

\def\R{\mathbb R}

\def\int{{\rm int\,}}

\def\eps{\varepsilon}

\newtheorem{theorem}{Theorem}[section]
\newtheorem{corollary}{Corollary}[section]
\newtheorem{lemma}{Lemma}[section]
\newtheorem{proposition}{Proposition}[section]
\theoremstyle{definition}
\newtheorem{definition}{Definition}[section]
\newtheorem{example}{Example}[section]
\newtheorem{remark}{Remark}[section]

\numberwithin{equation}{section}
\begin{document}
	\setcounter{page}{1}

	\vspace*{1.0cm}
	\title[Global Finite-time and Global Fixed-time  stable Dynamical Systems  for IQVIPs]
	{Global Finite-time and Global Fixed-time stable Dynamical Systems  for solving Inverse Quasi-variational inequality problems}
	\author[N.V. Tran, L.T.T Hai]{Nam V. Tran$^{3,*}$, L.T.T Hai$^{1,2, 3}$}
	\maketitle
	\vspace*{-0.6cm}

	\begin{center}
		{\footnotesize  {\it
				$^1$Faculty of Mathematics and Computer Science, University of Science, Ho Chi Minh City, Vietnam; 
				
				$^2$Vietnam National University, Ho Chi Minh City, Vietnam; 
				
				$^3$Faculty of Applied Sciences, HCMC University of Technology and Education, Ho Chi Minh City, Vietnam; }}

		
	\end{center}

	\vskip 4mm {\small\noindent {\bf Abstract.}
		In this paper, we propose two projection dynamical systems for solving inverse quasi-variational inequality problems in finite-dimensional Hilbert spaces—one ensuring finite-time stability and the other guaranteeing fixed-time stability. We first establish the connection between these dynamical systems and the solutions of inverse quasi-variational problems. Then, under mild conditions on the operators and parameters, we analyze the global finite-time and global fixed-time stability of the proposed systems. Both approaches offer accelerated convergence; however, while the settling time of a finite-time stable dynamical system depends on initial conditions, the fixed-time stable system achieves convergence within a predefined time, independent of initial conditions. To demonstrate their effectiveness, we provide numerical experiments, including an application to the traffic assignment problem.

		 }
	
	\noindent {\bf[MSC Classification]} {47J20; 47J30;49J40; 49J53; 49M37}

	\renewcommand{\thefootnote}{}
	\footnotetext{ $^*$Corresponding author.
		\par
		E-mail addresses: namtv@hcmute.edu.vn (Nam V. Tran),   hailtt@hcmute.edu.vn (Hai T.T. Le).
		\par
	}
	\section{Introduction} 
	We denote by $\langle \cdot,\cdot \rangle$ the inner product and by $\|\cdot\|$ the norm on $\R^n$. This paper focuses on the following 
    \textcolor{blue}{inverse quasi-variational inequality problem} (IQVIP):
Find $u^*\in \R^n$
 such that 
 \begin{equation}\label{iqvip}
     f(u^*)\in \Phi(u^*) \mbox { and } \langle u^*, y-f(u^*) \rangle \geq 0 \quad \forall y\in \Phi(u^*)
 \end{equation}
	where $\Phi: \R^n \longrightarrow 2^{\R^n}$ is a multi-valued operator and $f:\R^n \longrightarrow \R^n $ is a single-valued operator.

A notable special case arises when $\Phi(u)=\Omega$  is a fixed, nonempty, closed, and convex set, independent of $u$. In this case, the IQVIP \eqref{iqvip} reduces to the inverse variational inequality problem (IVIP), formulated as follows:
  Find $u\in \R^*$ such that 
  \begin{equation}\label{ivip}
     f(u^*)\in \Omega \mbox { and } \langle u^*, y-f(u^*) \rangle \geq 0 \quad \forall y\in \Omega .
 \end{equation} 
This problem has numerous applications across various fields (see, for instance, \cite{AUS, HAN}).

Furthermore, the IVIP \eqref{ivip} can be interpreted as a standard variational inequality problem when $f$ 
 is the inverse of a given mapping $F$.  Specifically, if $f(u)=F^{-1}(u)=u$ then the IVIP \eqref{ivip} can be rewritten as:

  Find $u\in \Omega$ such that 
  \begin{equation*} 
      \langle F(u^*) , v-u^*\rangle \geq 0, \quad \forall v\in \Omega.
  \end{equation*}

This formulation serves as a fundamental mathematical model that encompasses a broad spectrum of problems, including optimization, variational inequalities, saddle point problems, Nash equilibrium problems in noncooperative games, and fixed-point problems. For further details, see \cite{BlumOettli94, Cav} and the references therein.

The concept of finite-time convergence, introduced in \cite{BHAT}, has inspired the development of various dynamical systems exhibiting this property (see \cite{JU6, LI}). In such systems, the settling time depends on the initial state and may increase indefinitely as the initial deviation from equilibrium grows. However, in many practical applications—such as robotics and vehicle monitoring networks—precise knowledge of the initial state is often unavailable, making it difficult to determine the settling time in advance.

To overcome this drawback,  Polyakov in \cite{polyakov} introduced the notion of fixed-time convergence, where the convergence time is bounded by a constant independent of the initial state. This idea has since been further explored in studies such as \cite{Garg21, JU3, SAN, WANG, WANG2, WANG3, ZUO}.

Both finite-time and fixed-time convergence principles have been applied in various fields such as optimization problems \cite{BEN, COR, GAR, HE, LIEN, ROMEO}, sparse optimization \cite{GARG4, HE,  YU}, control problems, including multi-agent control \cite{ZUO}, Mixed variational inequality problems (MVIPs)  \cite{Garg21, JU3}, Variational inequality problems (VIPs) and Equilibrium problems (EPs) \cite{JU4, JU5}. More recently, Nam et al. \cite{NAM_HAI} studied finite-time stability for generalized monotone inclusion problems, while in \cite{NHVA}, they developed a fixed-time stable dynamical system for generalized monotone inclusion problems. Despite these advancements in dynamical systems these problems, their application to inverse quasi-variational inequality problems (IQVIPs) remains largely unexplored. In particular, the development of dynamical systems for IQVIPs remains unexplored, motivating this study.

In addition to discrete-time algorithms, continuous-time dynamical systems have gained popularity due to their effectiveness and computational efficiency in solving various problems, including inverse quasi-variational inequalities (IQVIPs), mixed variational inequalities (MVIs), variational inequalities (VIs), and constrained optimization problems (COPs) \cite{AAS, BC_SICON, 22, BCJMAA, BotCV, 18, 21, 27, 26, 20, JU4, LIU, 29, VS20, ZHU}. As for inverse quasi-variational inequality problems, there are some works using dynamical systems to study it, for instance, \cite{ZOU2}. In recently, (2023), Dey and Reich  \cite{Dey} employed a first-order dynamical system to analyze IQVIPs and established a globally exponential stability for these problems. In \cite{VuongThanh}, the authors obtained linear convergence for the sequence generated by the discretized time of a first-order dynamical system, as in \cite{Dey}, but for a more general moving set condition \(\Phi(u)\).    

Despite the extensive use of dynamical systems, most theoretical research has primarily focused on asymptotic stability \cite{27, 26, 29} or exponential stability \cite{22, 18, 21, 20, JU, JU4, LIU, V, VS20, ZHU}.

In this paper, we build upon this research by proposing two first-order projection dynamical systems based on a fixed-point reformulation (see, e.g., \cite{AM, Bot, BSV, V0}). We establish both finite-time and fixed-time stability for these systems. Additionally, by using forward Euler time discretization we obtain a projection-type method. We further demonstrate that the iterative sequence generated by this algorithm exhibits linear convergence to the unique solution of the 
\textcolor{blue}{inverse quasi-variational
inequality problem.}

The structure of the paper is as follows:

Section \ref{Preliminaries} reviews fundamental definitions and concepts, including projection dynamical systems and the notions of finite-time and fixed-time stability. We also introduce a nominal projection dynamical system associated with inverse quasi-variational inequality problems (IQVIPs) and establish the connection between its equilibrium points and the solutions of IQVIPs. Additionally, we present some new technical lemmas.

Section \ref{finite stability} introduces a newly proposed finite-time projection dynamical system model and provides its theoretical analysis.

Section \ref{sec3} focuses on analyzing the fixed-time stability of another projection dynamical system designed for solving IQVIPs.

Section \ref{sec4} explores a sufficient condition for achieving a consistent discretization, ensuring that the discretized fixed-time stable dynamical system converges within a fixed number of time steps. As a specific case, we demonstrate that the forward-Euler discretization of the modified projection dynamical system satisfies this consistency condition.

Section \ref{num} presents numerical experiments to demonstrate the effectiveness and advantages of the proposed algorithm, which is derived by discretizing the continuous-time fixed-time stable dynamical system.

Section \ref{sec6} concludes the paper with final remarks.
	

	\section{Preliminaries} \label{Preliminaries}
	
	In this section, we recall some well-known definitions useful in the sequel.

	\subsection{Some notions on convex analysis} 
	
	An operator $T: \R^n  \to \R^n$ is said to be $L$-Lipschitz continuous on \(\R^n\) if 
    \begin{equation*}
        \|T(u)-T(v)\|\leq L\|u-v\|, \quad \mbox{ for all } x, y\in \R^n.
    \end{equation*} 
    When $L=1$, we say that $T$ is {\it nonexpansive.}
	
	\noindent An operator  $T: \R^n \to \R^n$ is said to be \emph{ strongly monotone} with modulus $\beta$ if for any $u, \textcolor{blue}{v}\in \R^n$ we have 
    \begin{equation*}\label{mon}
        \langle f(u)-f(v), u-v\rangle \geq \beta \|u-v\|^2.
    \end{equation*}
	When $\beta=0$ we say that $T$ is monotone. 

    Note that if $T$ is strongly monotone with modules $\beta\geq 0$ then we have 
    \begin{equation*}
        \|f(u)-f(v)\|\geq \beta \|u-v\|, \quad \mbox { for all } \textcolor{blue} {u, v}\in \R^n.
    \end{equation*}

	\noindent For every $u\in \R^n$, the metric projection $P_C(u)$ of $u$ onto $C$ is defined by
	$$
	P_C (u)=\arg\min\left\{\left\|v-u\right\|:v\in C\right\}.
	$$
	It is worth noting that when $C$ is nonempty, closed, and convex, $P_C (u)$ exists and is unique. 
Some useful properties of the projection operator are listed below.
\begin{lemma}

Let $C$ be a nonempty, closed, convex set in \(\R^n\). Let $u, y \in \R^n$. Then  one has
\begin{enumerate}
    \item $   v=P_C(u) \iff \langle v-u, w-u\rangle \leq 0 \ \ \mbox{ for all } w\in C$
\item $P_C$ is 1-Lipschitz continuous on $C$.
\end{enumerate}

\end{lemma}
	\subsection{Equilibrium points and stability}
Consider a general dynamical system
	\begin{equation}\label{JJS1}
		\dot{u}(t) = T(u(t)), \quad t \ge 0,
	\end{equation}
	where $T$ is a continuous mapping from $\R^n$ to $\R^n$ and $u:[0, +\infty)\to \R^n$.

In this paper we investigate finite-time and fixed-time stability of equilibrium points of dynamical systems; we recall these definitions below.  
	\begin{definition} \cite{Pappaladro02}	 A point $u^*\in \R^n$ is said to be an equilibrium point for (\ref{JJS1}) if $T(u^*)=0$; An equilibrium point is called 
    \begin{enumerate}[(a)]
			\item {\emph{finite-time stable}} if it is stable in the sense of Lyapunov, that is, for any $\epsilon>0$, there exists  $\delta>0$ such that, for every $u_0 \in B(u^*, \delta) $, the solution $u(t)$ of the dynamical system with $u(0)=u_0$ exists and is contained in  $B(u^*, \epsilon)$ for all $t>0$,  and there exists a neighborhood \(B(\delta, u^*)\) of \(u^*\) and a settling-time function \(T: B(\delta, u^*)\setminus \{u^*\}\rightarrow (0, \infty)\) such that for any \(u(0)\in  B(\delta, u^*)\setminus\{u^*\}\), the solution of \eqref{JJS1} satisfies \(u(t)\in B(\delta, u^*)\setminus \{u^*\}\) for all \(t\in [0, T(u(0)))\) and \(\lim_{t\to T(u(0))}u(t)=u^*\).
			\item \emph{Globally finite-time stable}  if  it is finite-time stable with \(B(\delta, u^*)=\R^n\).
			\item \emph{ Fixed-time stable} if it is globally finite-time stable, and the settling-time function satisfies
			\[\sup_{u(0)\in \R^n}T(u(0))<\infty.\]
		\end{enumerate}

	\end{definition}
here  $B(\textcolor{blue} \delta, u^*)$ denotes the open ball with center $u^*$ and radius $\delta$.
		
From now on, we denote by $\mbox{\rm Zer}(T)$ the set of solutions of equation $T(u)=0$. 
\subsection{A nominal dynamical system for inverse quasi-variational inequality problems }
In this section we first recall a nominal dynamical system proposed in \cite{Dey} for IQVIPs \eqref{iqvip} and a result on the existence and uniqueness of solutions of IQVIPs \eqref{iqvip}. We then present several auxiliary results that will be useful in the subsequent analysis. In particular, we establish key inequalities that play a crucial role in studying finite-time and fixed-time stability.  

In 2019, \cite{ZOU2} Zou et al.  proposed the following dynamical system for solving inverse quasi-variational inequality problem \eqref{iqvip} 

\begin{equation}\label{firstdynamicalsystem1} 
	\dot{u}(t)=-\sigma \left(f(u(t))-P_{\Phi(u(t))}(f(u(t))-\alpha. u(t))\right), 
\end{equation}
where \(\sigma, \alpha\) are  positive constants. 

In 2023 \cite{Dey} Dey and Reich proposed the following 

\begin{equation}\label{firstdynamicalsystem} 
	\dot{u}(t)=-\sigma(t) \left(f(u(t))-P_{\Phi(u(t))}(f(u(t))-\alpha. u(t))\right), 
\end{equation}
where \(\sigma(t)>0\) for all $t>0$.
\begin{remark}\label{re 1}
As shown in \cite{Dey}, a necessary and sufficient condition for $u^*\in \R^n$ to be a solution of IQVIP \eqref{iqvip} is that it is also an equilibrium point of dynamical system \eqref{firstdynamicalsystem} or of \eqref{firstdynamicalsystem1}. 
\end{remark}

 To analyze the finite-time and fixed-time stability we make the following assumption:

({\bf A})  Let $\Phi: \R^n \longrightarrow 2^{\R^n} $  be a set-valued mapping with nonempty, closed, and convex values. Let   \(f: \R^n \longrightarrow \R^n\) be an $L-$Lipschitz continuous and \(\beta-\)strongly monotone single-valued mapping. Assume that there exists some \(\mu  > 0\) such that
\begin{equation}\label{dk phi}
    \|P_{\Phi(u)}(w) -P_{\Phi(v)}(w)\| \leq \mu \|u-v\|, \quad \forall u, v, w\in \R^n
\end{equation}
and
\begin{equation*}
    \sqrt{L^2+\alpha^2-2\alpha\beta}+\mu<\alpha.
\end{equation*}
We recall a result on the existence and uniqueness of solutions to the IQVIP \eqref{iqvip}. 
 \begin{theorem}\cite{Dey}\label{unique sol}
        Let $\Phi: \R^n \longrightarrow 2^{\R^n} $  and  \(f: \R^n \longrightarrow \R^n\) be such that assumption ({\bf A}) is satisfied. 
Then the inverse quasi-variational inequality problem \eqref{iqvip} has a unique solution.
    \end{theorem}

The following lemma lists some useful inequalities that are needed in the sequel. 

\begin{lemma}\label{bound for proj}
Let $\Phi: \R^n \longrightarrow 2^{\R^n} $  be a set-valued mapping with nonempty, closed and convex values, and  \(f: \R^n \longrightarrow \R^n\) be an $L-$Lipschitz continuous and
\(\beta-\)strongly monotone mapping. Assume further that there exists some \(\mu  > 0\) such that
\begin{equation}\label{dk phi 2}
    \|P_{\Phi(u)}(w) -P_{\Phi(v)}(w)\| \leq \mu \|u-v\|, \quad \forall u, v, w\in \R^n.
\end{equation}
Assume that the solution set of IQVIP \eqref{iqvip} is nonempty. Let $u^*\in \R^n$ be a solution of IQVIP \eqref{iqvip}. Then 
\begin{enumerate}
    \item \label{bound for proj p1} The inequality holds 
    \begin{equation*}
        \|P_{\Phi(u)}(f(u)-\alpha u) -P_{\Phi(u^*)}(f(u^*)-\alpha u^*)\| \leq \left(\sqrt{L^2+\alpha^2-2\alpha\beta}+\mu\right)\|u-u^*\|.
    \end{equation*}
    \item \label{bound fo iner} One has 
    \begin{align*}
		&\langle u-u^*, f(u)-P_{\Phi(u)}(f(u)-\alpha u) \rangle      
        \geq (\beta-\sqrt{L^2+\alpha^2-2\alpha \beta}-\mu) \|u-u^*\|^2 >0,
	\end{align*}
    for all  \(u\neq u^*.\)
     \item \label{bound for operator} It holds that 
     \begin{equation*}
    \|f(u)-P_{\Phi(u)}(f(u)-\alpha u)\|\leq \left(L+\mu+\sqrt{L^2+\alpha^2-2\alpha\beta}\right) \|u-u^*\|.
    \end{equation*} 
    \item \label{part4} The following holds
    \begin{equation*}
    \|f(u)-P_{\Phi(u)}(f(u)-\alpha u)\|\geq \left|\beta -\mu-\sqrt{L^2+\alpha^2-2\alpha\beta}\right| \|u-u^*\|.
    \end{equation*}  
\end{enumerate}  
\end{lemma}
\begin{proof}
\eqref{bound for proj p1} Since $f(u^*)=P_{\Phi(u^*)}(f(u^*)-\alpha u^*)$, it follows that 
    \begin{align*}
        & \|P_{\Phi(u)}(f(u)-\alpha u) -P_{\Phi(u^*)}(f(u^*)-\alpha u^*)\|\\
        =& \|P_{\Phi(u)}(f(u)-\alpha u) -P_{\Phi(u)}(f(u^*)-\alpha u^*)+P_{\Phi(u)}(f(u^*)-\alpha u^*)-P_{\Phi(u^*)}(f(u^*)-\alpha u^*)\|.
        \end{align*}
\noindent Using triangular inequality and by the non-expansion property of the projector and condition \eqref{dk phi 2}, we have that 
  \begin{align*} 
         & \|P_{\Phi(u)}(f(u)-\alpha u) -P_{\Phi(u^*)}(f(u^*)-\alpha u^*)\| \\
     \leq & \|P_{\Phi(u)}(f(u)-\alpha u) -P_{\Phi(u)}(f(u^*)-\alpha u^*)\|+ \|P_{\Phi(u)}(f(u^*)-\alpha u^*)-P_{\Phi(u^*)}(f(u^*)-\alpha u^*)\|\\ 
        \leq & \|f(u)-f(u^*)-\alpha(u-u^*)\|+\mu\|u-u^*\| \notag \\
         = & \sqrt{\|f(u)-f(u^*)\|^2+\alpha^2\|u-u^*\|^2-2 \langle u-u^*, f(u)-f(u^*)\rangle}+\mu\|u-u^*\|\\
         \leq &\sqrt{L^2\|u-u^*\|\textcolor{blue}{^2}+\alpha^2\|u-u^*\|^2-2\alpha\beta\|u-u^*\|^2}+\mu\|u-u^*\| \notag\\
         = &\left(\sqrt{ (L^2+\alpha^2-2\alpha\beta)}+\mu \right)\|u-u^*\|. 
    \end{align*}
\eqref{bound fo iner}	For all \(u\in  \R^n\setminus \{u^*\},\) by Part \ref{bound for proj p1} of Lemma \ref{bound for proj} it holds that 
	\begin{align*}
		&\langle u-u^*, f(u)-P_{\Phi(u)}(f(u)-\alpha u) \rangle\\
        =& \langle u-u^*, f(u)-f(u^*)-(P_{\Phi(u)}(f(u)-\alpha u)-P_{\Phi(u^*)}(f(u^*)-\alpha u^*))\rangle\\
        =& \langle u-u^*, f(u)-f(u^*\rangle -\langle u-u^*, P_{\Phi(u)}(f(u)-\alpha u)-P_{\Phi(u^*)}(f(u^*)-\alpha u^*)\rangle\\
        \geq  & \beta\|u-u^*\|^2 - \|u-u^*\|\|P_{\Phi(u)}(f(u)-\alpha u)-P_{\Phi(u^*)}(f(u^*)-\alpha u^*)\|\\
        \geq  &\beta\|u-u^*\|^2 - (\sqrt{L^2+\alpha^2-2\alpha \beta}+\mu)\|u-u^*\|^2\\
        = & (\beta-\sqrt{L^2+\alpha^2-2\alpha \beta}-\mu) \|u-u^*\|^2 >0 \mbox{ for all } u\neq u^*.
	\end{align*}
	\eqref{bound for operator}  We have 
    \begin{align*}
        \|f(u)-P_{\Phi(u)}(f(u)-\alpha u)\| & = \|f(u)-f(u^*)- (P_{\Phi(u)}(f(u)-\alpha u)-P_{\Phi(u^*)}(f(u^*)-\alpha u^*))\|\\
        & \leq \|f(u)-f(u^*)\| + \|P_{\Phi(u)}(f(u)-\alpha u)-P_{\Phi(u^*)}(f(u^*)-\alpha u^*)\| \\
        & \leq L\|u-u^*\|+\sqrt{L^2+\alpha^2-2\alpha\beta}\|u-u^*\|\\
        &=\left(L+\mu+\sqrt{L^2+\alpha^2-2\alpha\beta}\right) \|u-u^*\|.
    \end{align*}
    \eqref{part4} By Part \ref{bound for proj p1} we have  \begin{align*}
        \|f(u)-P_{\Phi(u)}(f(u)-\alpha u)\| & = \|f(u)-f(u^*)- (P_{\Phi(u)}(f(u)-\alpha u)-P_{\Phi(u^*)}(f(u^*)-\alpha u^*))\|\\
        & \geq \left|\|f(u)-f(u^*)\| - \|P_{\Phi(u)}(f(u)-\alpha u)-P_{\Phi(u^*)}(f(u^*)-\alpha u^*)\| \right| \\
        & \geq \left| \beta\|u-u^*\| - \left(\sqrt{L^2+\alpha^2-2\alpha\beta}+\mu\right)\right|\|u-u^*\|\\
        &=\left|\beta- \left(\mu+\sqrt{L^2+\alpha^2-2\alpha\beta}\right)\right| \|u-u^*\|.
    \end{align*}

\end{proof}

\begin{lemma}\label{LL cont}
Let $\Phi: \R^n \longrightarrow 2^{\R^n} $  be a set-valued mapping with nonempty, closed and convex  values, and  \(f: \R^n \longrightarrow \R^n\) be an $L-$Lipschitz continuous and \(\beta-\)strongly monotone mapping. Assume further that there exists some \(\mu  > 0\) such that
\begin{equation}\label{dk phi 1}
    \|P_{\Phi(u)}(w) -P_{\Phi(v)}(w)\| \leq \mu \|u-v\|, \quad \forall u, v, w\in \R^n.
\end{equation}
Then the operator $T(u)\equiv f(u)- P_{\Phi(u)}(f(u)-\alpha u)$ is Lipschit continuous. 
\end{lemma}
\begin{proof} By the non-expansion property of the projection operator and \eqref{dk phi 1} we have that  
    \begin{align*} 
    \|T(u)-T(v)\|=& \| f(u)- P_{\Phi(u)}(f(u)-\alpha u) -(f(v)- P_{\Phi(v)}(f(v)-\alpha v))\|\\
    \leq & \|f(u)-f(v)\|\\
    +&\|P_{\Phi(u)}(f(u)-\alpha u)-P_{\Phi(u)}(f(v)-\alpha v)+P_{\Phi(u)}(f(v)-\alpha v)-P_{\Phi(v)}(f(v)-\alpha v))\\
    \leq & L\|u-v\| + \|f(u)-f(v)\|+\alpha\|u-v\|+\mu \|u-v\|\\
    \leq & (2L+\alpha+\mu)\|u-v\|.
    \end{align*} 
    This implies that $T$ is Lipschit continuous.
\end{proof}

\section{Finite-time stability Analysis}\label{finite stability}
In this section we will analyze the finite-time stability for IQVIPs \eqref{iqvip}. By doing so we propose a new first order dynamical system associated with the problem \eqref{iqvip} such that a solution of \eqref{iqvip} becomes an equilibrium point of the dynamical system. Under mild conditions for the parameters, we demonstrate that the proposed dynamical system is finite-time stable.

We first recall a characterization for the finite-time stability of solutions of a dynamical system, given by  Bhat and Bernstein in  \cite{BHAT}. This result will play an important role in studying finite-time convergence.

\begin{theorem}\cite{BHAT} \label{lm1-finite}	(Lyapunov condition for finite-time stability). Assume that there exists 	a continuously differentiable function $V : D\rightarrow \R$, where $D \subseteq \R^n$  is a neighborhood of the	equilibrium point \(u^*\) for \eqref{JJS1}  and an open neighborhood $U\subseteq D$ of $u^*$ such that
\begin{equation*}
	\dot{V}(u) \leq -K.\big(V(u)\big)^p \quad \forall u\in U\setminus \{u^*\},
\end{equation*}
where $K>0$ and $p\in (0,1).$  Then, the equilibrium point \(u^*\) of \eqref{JJS1} is finite-time stable equilibrium point of \eqref{JJS1}. Moreover, the settling time $T$ satisfies 

\begin{equation*}
	T(u(0)) \leq \dfrac{V(u(0))^{1-p}}{K(1-p)} 
\end{equation*}
for any $u(0) \in U$.  

In addition, if $D =\R^n$, then the equilibrium point \(u^*\) of \eqref{JJS1} is globally finite-time stable.
\end{theorem}

In order to obtain a finite-time stability of solutions of inverse quasi-variational inequality problem \eqref{iqvip} we now present a novel first-order dynamical system associated with the problem \eqref{iqvip}.  Our proposed dynamical system is as follows: 
\begin{equation}\label{newydynamicalsystem-finite} 
\dot{u}(t)=-\sigma  \dfrac{f(u)-P_{\Phi(u)}(f(u)-\alpha u)}{\|f(u)-P_{\Phi(u)}(f(u)-\alpha u)\|^{\frac{\gamma-2}{\gamma-1}}}, 
\end{equation} 
where 
$\sigma>0$  is a scalar tuning gain, and $\gamma>2, \alpha>0$ are two design parameters. 

It is worth noting that for all $u\in \R^n$ such that $f(u)- P_{\Phi(u)}(f(u)-\alpha u)=0$, the expression on the right-hand side of \eqref{newydynamicalsystem-finite} is still well defined for all $\gamma>2$. In fact, it is equal to zero. Indeed,  we have 
\begin{equation*}
   \left\|  \dfrac{f(u)-P_{\Phi(u)}(f(u)-\alpha u)}{\|f(u)-P_{\Phi(u)}(f(u)-\alpha u)\|^{\frac{\gamma-2}{\gamma-1}}}\right\|= \|f(u)-P_{\Phi(u)}(f(u)-\alpha u)\|^{\frac{1}{\gamma-1}}=0.
\end{equation*}

The following result establishes the connection between the equilibrium points of the previously mentioned dynamical system and those of \eqref{firstdynamicalsystem}.
\begin{proposition}\label{pro3-finite}
A point $u^*\in \R^n$ is an equilibrium point of \eqref{newydynamicalsystem-finite} if it is also an equilibrium point of \eqref{firstdynamicalsystem}, and vice versa. 
\end{proposition}
\begin{proof} Observe that  $u^*\in \R^n$ is an equilibrium point of \eqref{newydynamicalsystem-finite} if and only if 
$$\dfrac{f(u^*)-P_{\Phi(u)}(f(u^*)-\alpha u^*)}{\|f(u^*)-P_{\Phi(u)}(f(u^*)-\alpha u^*)\|^{\frac{\gamma-2}{\gamma-1}}}= 0.$$  It holds if and only if   $$  \dfrac{\|f(u^*)-P_{\Phi(u)}(f(u^*)-\alpha u^*)\|}{\|f(u^*)-P_{\Phi(u)}(f(u^*)-\alpha u^*)\|^{\frac{\gamma-2}{\gamma-1}}} =0.$$ Again, this equation is fulfilled if and only if 
\begin{equation}\label{equ80}
	\|f(u^*)-P_{\Phi(u)}(f(u^*)-\alpha u^*)\|^{1-\frac{\gamma -2}{\gamma -1}}=0.
\end{equation}
Because $1-\frac{\gamma -2}{\gamma -1}=\frac{1}{\gamma-1}>0$ for all $\gamma>2$, the equality \eqref{equ80} holds if and only if  $$f(u^*)-P_{\Phi(u)}(f(u^*)-\alpha u^*)=0,$$ or $$f(u^*)=P_{\Phi(u)}(f(u^*)-\alpha u^*).$$ This holds if and only if  $u^*$ is an equilibrium point of \eqref{firstdynamicalsystem}. The proof is thus completed.  
\end{proof}

From Proposition \ref{pro3-finite}, we derive the following direct consequence:
\begin{corollary}\label{pro4} 
	Let $\Phi: \R^n \longrightarrow 2^{\R^n} $  be a set-valued mapping with nonempty, closed, and convex values, and \(f: \R^n \longrightarrow \R^n\) be a single mapping.  Then \(u^*\) is a solution to the inverse quasi-variational inequality problem \eqref{iqvip} if and only if it is an equilibrium point of the dynamical system \eqref{newydynamicalsystem-finite}. 
\end{corollary}
\begin{proof}
It is a direct consequence of  Proposition \ref{pro3-finite} and Remark \ref{re 1}. 
\end{proof}

The next main result gives the finite-time stability of the dynamical system \eqref{newydynamicalsystem-finite}.

\begin{theorem}\label{tr2-fin}  Let $\Phi: \R^n \longrightarrow 2^{\R^n} $  be a set-valued mapping with nonempty, closed, and convex values and \(f: \R^n \longrightarrow \R^n\) be such that conditions in assumption {\bf (A)} hold. Assume further that 
\begin{equation}\label{dk cho beta}
    \sqrt{L^2+\alpha^2-2\alpha\beta}+\mu<\beta.
\end{equation}
Then,  the solution \(u^*\in \R^n\) of the inverse quasi-variational inequality problem  \eqref{iqvip} is a globally finite-time stable equilibrium point of \eqref{newydynamicalsystem-finite} with a settling time 
$$T(u(0))\leq T_{\max}=\dfrac{1}{2^{1-p}}\dfrac{(\|u(0)-u^*\|)^{2(1-p)}}{K(1-p)}$$
for some constants $K>0$ and $p \in (0.5, 1)$. 
\end{theorem}
\begin{proof} 

Consider the Lyapunov function \(V : \R^n \rightarrow \R\) defined as follows:
\[V(u):=\dfrac{1}{2}\|u-u^*\|^2.\]
Differentiating  with respect to time of the Lyapunov function \(V\) along the solution of \eqref{newydynamicalsystem-finite}, starting from any initial condition \(u(0)\in \R^n\setminus\{u^*\}\), where $u^*\in \mbox{\rm Zer}(T)$ is unique with $T(u)=f(u)-P_{\Phi(u)}(f(u)-\alpha u)$,   yields:
\begin{align*}
	\dot{V}=&-\left\langle u-u^*, \sigma  \dfrac{f(u)-P_{\Phi(u)}(f(u)-\alpha u)}{\|f(u)-P_{\Phi(u)}(f(u)-\alpha u)\|^{\frac{\gamma-2}{\gamma-1}}} \right\rangle \notag\\
	=& -\sigma   \dfrac{\left\langle u-u^*, f(u)-P_{\Phi(u)}(f(u)-\alpha u) \right\rangle} {\|f(u)-P_{\Phi(u)}(f(u)-\alpha u)\|^{\frac{\gamma-2}{\gamma-1}}}. 
\end{align*}
By using Part \ref{bound fo iner} of Lemma \ref{bound for proj} it follows that 
\begin{align*}
	\dot{V} &\leq  -\sigma (\beta-\sqrt{L^2+\alpha^2-2\alpha\beta}-\mu)   \dfrac{\| u-u^*\|^2} {\|f(u)-P_{\Phi(u)}(f(u)-\alpha u)\|^{\frac{\gamma-2}{\gamma-1}}} 
\end{align*}
for all $u\in \R^n\setminus \{u^*\}$. 

By using Part \ref{part4} of Lemma \ref{bound for proj} we have that 
\begin{align}\label{eq16-finite1}
	\dot{V} \leq &  -\sigma \dfrac{ (\beta-\sqrt{L^2+\alpha^2-2\alpha\beta}-\mu)}{(\beta-\sqrt{L^2+\alpha^2-2\alpha\beta}-\mu)^{\frac{\gamma-2}{\gamma-1}} }  \dfrac{\| u-u^*\|^2} { \|u-u^*\|^{\frac{\gamma-2}{\gamma-1}}}\\ \notag
    =&-\sigma (\beta-\sqrt{L^2+\alpha^2-2\alpha\beta}-\mu)^{\frac{\gamma}{\gamma-1}} \|u-u^*\|^{\frac{\gamma}{\gamma-1}}.\label{eq16-finite1} 
\end{align}
Setting \(M:= \sigma (\beta-\sqrt{L^2+\alpha^2-2\alpha\beta}-\mu)^{\frac{\gamma}{\gamma-1}}\). Then by condition \eqref{dk cho beta} it holds that  $M>0$ and inequality  \eqref{eq16-finite1} can be read as. 
\begin{align*}
	\dot{V} \leq -M  \|u-u^*\|^{\frac{\gamma}{\gamma-1}} =- M.2^{\frac{\gamma}{2(\gamma-1)}} \left(\dfrac{1}{2}\|u(0)-u^*\|^2\right)^{\frac{\gamma}{2(\gamma-1)}}.
\end{align*}
By applying  Theorem \ref{lm1-finite} with noting that $0.5<p=\frac{1}{2} \cdot \frac{\gamma}{\gamma-1}<1$ for all $\gamma>2$ and $K=M.2^{p}>0$ we obtain the conclusion. The proof is completed.
\end{proof}
\section{ Fixed-time stability Analysis} 	\label{sec3}
This section is for analyzing the fixed-time stability of solutions of IQVIPs \eqref{iqvip}.  We will introduce another projection dynamical system for the problem \eqref{iqvip}.  Next, we characterize the solutions of inverse quasi-variational inequality problems through equilibrium points of the proposed dynamical system. Finally, we establish the global fixed-time convergence of the proposed dynamical system.

Again, we recall here a result on fixed-stable time stability given by  Polyakov obtained in  \cite{polyakov}, which plays a fundamental role in analyzing fixed-time stability. 
	\begin{theorem} \label{lm1}	(Lyapunov condition for fixed-time stability). Assume that there exists 	a continuously differentiable function $V : D\rightarrow \R$, where $D \subseteq H$  is a neighborhood of the	equilibrium point \(u^*\) for \eqref{JJS1} such that
		\begin{equation*}
			V(u^*) = 0, V(u) > 0
		\end{equation*}
		for all $u \in  D \setminus \{u^*\}$ and
		\begin{equation} \label{inq of Lyapunov} 
			\dot{V}(u) \leq  -(s_1V(u)^{p_1} +s_2V(u)^{p_2})^{p_3}
		\end{equation}
		for all $u \in D \setminus \{u^*\}$ with $s_1,s_2,p_1,p_2,p_3 > 0$ such that $p_1p_3 <1$ and $p_2p_3 >1.$  Then, the equilibrium point \(u^*\) of \eqref{JJS1} is fixed-time stable with settling time function 
		\begin{equation*}\label{time1}
			T(u(0)) \leq \dfrac{1}{s_1^{p_3}(1-p_1p_3)} +\dfrac{1}{s_2^{p_3}(p_2p_3-1)}
		\end{equation*}
		for any $u(0) \in H$.  
		Moreover, take $p_1=\left(1-\frac{1}{2\nu}\right), p_2=\left(1+\frac{1}{2\nu}\right),$ with $\nu>1$ and $p_3=1$ in \eqref{inq of Lyapunov}. Then the equilibrium point of \eqref{JJS1} is fixed-time stable with the settling time 
		\begin{equation*}\label{time2}
			T(u(0))\leq T_{\max } =\dfrac{\pi\nu}{\sqrt{s_1s_2}}.
		\end{equation*}

		In addition, if the function $V$ is radially unbounded (i.e., $||u||\to \infty \Rightarrow V(u)\to \infty$) and $D =\R^n$, then the equilibrium point \(u^*\) of \eqref{JJS1} is globally fixed-time stable.
	\end{theorem}


To obtain the fixed-time stability for solutions to inverse quasi-variational inequality problems  \eqref{iqvip} we construct the following dynamical system:
\begin{equation}\label{newydynamicalsystem} 
	\dot{u}=-\psi(u)(f(u)-P_{\Phi(u)}(f(u)-\alpha u)),
\end{equation} 
where $\alpha>0$ is any positive number and 
\begin{equation}\label{dnrho} 
	\psi(u):=\begin{cases} 
		a_1\dfrac{1}{\|f(u)-P_{\Phi(u)}(f(u)-\alpha u)\|^{1-r_1}}& +a_2\dfrac{1}{\|f(u)-P_{\Phi(u)}(f(u)-\alpha u)\|^{1-r_2}} \\
		& \mbox{ if } u\in  \R^n\setminus {\rm Zer}(T)\\
		0 & \mbox{otherwise} 
	\end{cases}
\end{equation} 
with \(a_1, a_2>0, r_1\in (0, 1)\) and \(r_2>1\), $T(\cdot)=f(\cdot)-P_\Phi(\cdot )(f(\cdot)- \alpha. id(\cdot)),$ here $id(\cdot)$ is the identity operator. 
\begin{remark}\label{pro3}
A point $ u^*\in  \R^n$ is an equilibrium point of \eqref{newydynamicalsystem} if it is also an equilibrium point of \eqref{firstdynamicalsystem}, and vice versa. 
\end{remark}
Indeed,  because  \(\psi(u)=0\) if and only if \(u\in  \text{\rm Zer}(T)\), the conclusion follows from the definition \eqref{dnrho} of $\psi$. 

Combining  Remark \ref{re 1} and Remark \ref{pro3}  we have a characterization for solutions of the inverse quasi-variational inequality problem \eqref{iqvip} through equilibrium points of the dynamical system \eqref{newydynamicalsystem} as follows:

\begin{remark}\label{pro4-fixed}
	A point \(u^*\in \R^n\) is a solution of inverse quasi-variational inequality problem \eqref{iqvip} if and only if it is an equilibrium point of the dynamical system \eqref{newydynamicalsystem}. 
\end{remark}

The following lemma gives conditions so that solutions of the given dynamical system exist and are uniquely determined in the classical sense.

\begin{lemma}\cite{Garg21}\label{prop.1}
	Let $T: \R^n \rightarrow \R^n$ be a locally Lipschitz continuous vector field such that 
	$$	T(\bar{u}) = 0 \mbox{ and } \langle u -\bar{u}, T(u) \rangle  > 0$$
	for all $u \in  \R^n\setminus\{\bar{u}\}$. Consider the following autonomous differential equation:
	\begin{equation}\label{system3}
		\dot{u}(t)=-c(u(t)) T(u(t)),
	\end{equation}
	where 
	\begin{equation*}
		c(u):=\begin{cases}
			a_1\dfrac{1}{\|T(u)\|^{1-r_1}}+a_2\dfrac{1}{\|T(u))\|^{1-r_2}} & \mbox{ if } T(u)\neq 0\\ 
			0 & \mbox{otherwise} 
		\end{cases}
	\end{equation*} 
	with \(a_1, a_2>0, r_1\in (0, 1)\) and \(r_2>1\). Then, with any given initial condition, the solution of \eqref{system3} exists in the classical sense and is uniquely determined for all $t\geq 0$. 
\end{lemma}
We now show that solutions to the dynamical system \eqref{newydynamicalsystem} exist and are uniquely determined. 
\begin{proposition}\label{uunique sol for DS}
    Let $\Phi: \R^n \longrightarrow 2^{\R^n} $   and  \(f: \R^n \longrightarrow \R^n\) be such that assumption ({\bf A}) holds.
Then the dynamical system \eqref{newydynamicalsystem} has a unique solution in a classical sense. 
\end{proposition}
\begin{proof}

	By Theorem \ref{unique sol}, IQVIP \eqref{iqvip} has a unique solution. It follows that the set \(\text{\rm Zer}(T)\)  is a singleton, here  $T(u)\equiv f(u)-P_{\Phi(u)}(f(u)-\alpha u) $. Consequently,  the vector field in \eqref{firstdynamicalsystem} has a unique equilibrium point \(\overline{u}=u^*\). By  Lemma \ref{bound for proj}, Part \ref{bound fo iner}, the vector field $T(u)\equiv f(u)-P_{\Phi(u)}(f(u)-\alpha u) $ satisfies the following property:
	\begin{equation*}
		\langle u-\overline{u}, T(u) \rangle > 0
	\end{equation*}
	for all $u\in  \R^n\setminus \{\overline{u}\}$, where \(\{\overline{u}\}= \rm Zer(T)\).  Also, 
     by Lemma \ref{LL cont}, the operator \(T(\cdot)=P_{\Phi(\cdot)}(f(\cdot)-\alpha. id(\cdot))-f(\cdot)\) is Lipschitz continuous. This implies Lipschitz continuity on \(\R^n\) of the vector field on the right-hand side of \eqref{firstdynamicalsystem}. Finally, the conclusion follows from Lemma \ref{prop.1}. 
\end{proof}
We now present the main result of this section.

\begin{theorem}\label{tr2}   Let $\Phi: \R^n \longrightarrow 2^{\R^n} $  and   \(f: \R^n \longrightarrow \R^n\) be such that the conditions in assumption ({\bf A}) holds. Assume further that
\begin{align}\label{dk cho beta 2}
    \sqrt{L^2+\alpha^2-2\alpha\beta}+\mu<  \beta.
\end{align}
Let $u^*\in \R^n$ be an equilibrium point of dynamical system \eqref{newydynamicalsystem}.  Then the solution \(u^*\in \R^n\) of the IQVIP \eqref{iqvip} is a global fixed-time stable equilibrium point of \eqref{newydynamicalsystem} for any \(r_1\in (0, 1)\) and \(r_2>1\) and the following time estimate holds:
	$$T(u(0))\leq T_{\max}=\dfrac{1}{s_1(1-b_1)}+\dfrac{1}{s_2(b_2-1)}$$
	for some $p_1>0, p_2>0, b_1\in (0.5, 1), b_2>1$. 
	
	In addition, if take $b_1=1-\frac{1}{2\xi}, b_2=1+\frac{1}{2\xi}$ with $\xi>1$   then the following time estimate holds:
	$$T(u(0))\leq T_{\max} =\dfrac{\pi \xi}{\sqrt{s_1s_2}}$$
	for some constants $s_1>0, s_2>0$ and $\xi>1$. 
\end{theorem}
\begin{proof} By Proposition \ref{uunique sol for DS},  a solution of \eqref{newydynamicalsystem} exists and is uniquely determined for all forward times. Let  \(V: \R^n\rightarrow \R\) be the Lyapunov function defined by:
	\[V(u):=\dfrac{1}{2}\|u-u^*\|^2.\]
	Differentiating with respect to time along the solution of \eqref{newydynamicalsystem} for Lyapunov function, starting from any \(u(0)\in  \R^n \setminus\{u^*\}\) with noting that $u^*\in \text{\rm Zer}(f(\cdot)-P_\Phi(\cdot)(F(\cdot)-\alpha .id(\cdot))$ being unique  one has:
	\begin{align*}
		\dot{V}=&\langle u-u^*, \dot{u}\rangle \\
        = &- \dfrac{a_1}{\|f(u)-P_{\Phi(u)}(f(u)-\alpha u)\|^{1-r_1}} \langle u-u^*, f(u)-P_{\Phi(u)}(f(u)-\alpha u)\rangle \\
        &- \dfrac{a_2}{\|f(u)-P_{\Phi(u)}(f(u)-\alpha u)\|^{1-r_2}} \cdot \langle u-u^*, f(u)-P_{\Phi(u)}(f(u)-\alpha u)\rangle \notag	
        \end{align*}
	for all $u\in  \R^n\setminus \{u^*\}$. 
	By Lemma \ref{bound for proj}, Part \ref{bound fo iner} we derive that 
	\begin{align*}
		\dot{V}\leq  &- \dfrac{a_1}{\|f(u)-P_{\Phi(u)}(f(u)-\alpha u)\|^{1-r_1}} (\beta-\sqrt{L^2+\alpha^2-2\alpha \beta}-\mu) \|u-u^*\|^2\\
        &- \dfrac{a_2}{\|f(u)-P_{\Phi(u)}(f(u)-\alpha u)\|^{1-r_2}} \left(\beta-\sqrt{L^2+\alpha^2-2\alpha \beta}-\mu\right) \|u-u^*\|^2 \notag,
	\end{align*}
	for all $u\in  \R^n\setminus \{u^*\}$. 
	By Lemma \ref{bound for proj} again  one obtains that 
	\begin{align}\label{eq16fin} 
		\dot{V} \leq &-\left(\dfrac{a_1(\beta-\mu-\sqrt{L^2+\alpha^2-2\alpha\beta})}{(L+\mu+\sqrt{L^2+\alpha^2-2\alpha\beta})^{1-r_1}}\dfrac{\|u-u^*\|^2}{\|u-u^*\|^{1-r_1}}\right)\notag\\ 
        -&\left(\dfrac{a_2(\beta-\mu-\sqrt{L^2+\alpha^2-2\alpha\beta})}{(\beta -\mu-\sqrt{L^2+\alpha^2-2\alpha\beta})^{1-r_2}}\dfrac{\|u-u^*\|^2}{\|u-u^*\|^{1-r_2}}\right)\notag \\
		=&-q_1|u-u^*\|^{1+r_1}-q_2\|u-u^*\|^{1+r_2},
	\end{align}
	where \(q_1=\dfrac{a_1(\beta-\mu-\sqrt{L^2+\alpha^2-2\alpha\beta})}{(L+\mu+\sqrt{L^2+\alpha^2-2\alpha\beta})^{1-r_1}}\) and \(q_2= a_2(\beta-\mu-\sqrt{L^2+\alpha^2-2\alpha\beta})^{r_2}\). 
	By condition \eqref{dk cho beta 2} it holds that $q_1>0, q_2>0$ for all $0<r_1<1, r_2>1$.  Hence, \eqref{eq16fin} implies that 
	\begin{equation}\label{eq17} 
		\dot{V} \leq -\left( s_1V^{p_1}+s_2V^{p_2} \right),
	\end{equation}
	here \(s_i =2^{p_i}q_i\) and \(p_i=\frac{1+r_i}{2}\), for $i=1,2$. Note that \(s_1>0,p_1<1\) for any \(r_1\in (0, 1)\) and \(s_2>0, p_2>1\) for any \(r_2>1\). Then the conclusion follows from Theorem \ref{lm1}.  
\end{proof}

\section{Consistent discretization of the modified projection dynamical system} \label{sec4}
In this section, we show that the fixed-time stability of the dynamical system \eqref{newydynamicalsystem} is preserved under time discretization. It is worth noting that while continuous-time dynamical systems exhibit fixed-time convergence, this property does not always extend to their discrete-time counterparts. However, a well-designed discretization scheme can help maintain the convergence behavior of the continuous system in the discrete setting (see, for example, \cite{POL}). We prove that under the conditions of Theorem \ref{tr2}, a consistent discretization of the fixed-time convergent modified projection dynamical system \eqref{newydynamicalsystem} can be achieved.

We first recall a result related to the consistent discretization of a differential inclusion problem introduced in \cite{Garg21}. Then we apply this to our proposed dynamical system \eqref{newydynamicalsystem}. 

\begin{theorem}\cite{Garg21}\label{tr3} 
	Consider the following differential inclusion problem \begin{equation}\label{eq20} 
		\dot{u}\in \Upsilon(u)
	\end{equation}
	where \(\Upsilon: \R^n \longrightarrow 2^{\R^n}\) is an upper semi-continuous set-valued mapping such that its values are non-empty, convex, compact, and  \(0\in \Upsilon(\overline{u})\) for some \(\overline{u}\in  \R^n\). Assume further that there exists a positive definite, radially unbounded, locally Lipschitz continuous and regular function \(V: \R^n\rightarrow \R^n\) such that \(V(\overline{u})=0\) and 
	\begin{equation*}
		\sup \dot{V}(u)\leq -\left( s_1 V(u)^{1-\frac{1}{\nu}}+s_2 V(u)^{1+\frac{1}{\nu}}\right) 
	\end{equation*}
	for all \(u\in   \R^n\setminus\{\overline{u}\}\), with \(s_1, s_2>0\) and \(\nu>1\), here  
	\begin{equation*}
		\dot{V}(u)=\left\{ w\in \R: \exists u\in \Upsilon (u) \mbox{ such that } \langle z, u\rangle =w, \forall z\in \partial_cV(u)\right\}
	\end{equation*}
	and   \(\partial_c V(u)\) denotes Clarke's generalized gradient of the function \(V\) at the point \(u\in   \R^n\). Then, the equilibrium point \(\overline{u}\in  \R^n\) of \eqref{eq20} is fixed-time stable, with the settling-time function \(T\) satisfying 
	\begin{equation*}
		T(u(0)) \leq \dfrac{\nu\pi}{2\sqrt{s_1s_2}} 
	\end{equation*} 
	for any starting point  \(u(0)\in  \R^n\). 
\end{theorem}
By applying the forward-Euler discretization of \eqref{eq20} we get
\begin{equation}\label{eq24} 
	u_{n+1}\in u_n+\lambda \Upsilon(u_n),
\end{equation}
where \(\lambda >0\) is the time-step. 

Under additional conditions on the function $V$, the authors in \cite{Garg21} derive an upper bound for the error of the sequence generated by equation \eqref{eq24}.
\begin{theorem}\cite{Garg21}\label{tr4} 
	Assume that the conditions of Theorem \ref{tr3} are valid, and  the function \(V\) satisfies the following quadratic growth condition
	\begin{equation*}
		V(u)\geq c\|u-\overline{u}\|^2 
	\end{equation*}
	for every $u\in  \R^n$, where $c>0$ and 
	$\overline{u}$ 
	is the equilibrium point of \eqref{eq20}. 
	Then, for all \(u_0\in  \R^n\) and \(\eps>0\), there exists 
	 $\lambda^*>0$ such that for any \(\lambda\in (0, \lambda^*]\), one has:
	\begin{equation*}
		\|u_n-\overline{u}\| <\begin{cases}
			\frac{1}{\sqrt{c}} \left( \sqrt{\frac{s_1}{s_2}}\tan\left(\frac{\pi}{2}-\frac{\sqrt{s_1s_2}}{\nu}\lambda n\right)\right)^{\frac{\nu}{2}}+\eps, &\mbox{ if }\ n\leq n^*\\
			\eps & \mbox{otherwise},
		\end{cases}
	\end{equation*}
	where \(u_n\) is a solution  of \eqref{eq24} starting from the point \(u_0\)  and $n^* = \left \lceil \dfrac{\nu\pi}{2\lambda\sqrt{s_1s_2}} \right \rceil$.
\end{theorem}

Using this result we obtain the fixed-time stability for the discretization version of \eqref{newydynamicalsystem} and get an upper bound for settling time. 
\begin{theorem}
	Consider the forward-Euler discretization of \eqref{newydynamicalsystem}:
	\begin{equation}\label{eq29} 
		u_{n+1}=u_n-\lambda \psi(u_n)(f(u_n)-P_{\Phi(u_n)}(f(u_n)-\alpha u_n)),
	\end{equation}
	where  \(\psi\) is given by \eqref{dnrho},  \(a_1, a_2>0, r_1=1-2/\nu\) and \(r_2=1+2/\nu\), \(\nu\in (2, \infty)\), and \(\lambda>0\) is the time-step. Then for every \(u_0\in  \R^n\), every \(\eps>0\), if  conditions in assumption {\bf (A)} holds, there exist \(\nu>2, s_1, s_2>0\) and \(\lambda^*>0\) such that for any \(\lambda\in (0, \lambda^*]\), we have: 
	\begin{equation*}
		\|u_n-u^*\|< \begin{cases}
			\sqrt{2}\left( \sqrt{\frac{s_1}{s_2}}\tan\left(\frac{\pi}{2}-\frac{\sqrt{s_1s_2}}{\nu}\lambda n\right)\right)^{\frac{\nu}{2}}+\eps, &  \mbox{ if } n\leq n^*\\ 
			\eps & \mbox{otherwise},
		\end{cases}
	\end{equation*}
	where \(n^*=\left\lceil \dfrac{\nu\pi}{2\lambda \sqrt{s_1s_2}} \right\rceil\) and \(u_n\) is a solution of \eqref{eq29} starting from the point \(u_0\) and \(u^*\in  \R^n\) is the unique solution of inverse quasi-variational inequality problem \eqref{iqvip}. 
\end{theorem}
\begin{proof} From the proof of Theorem \ref{tr2} we observe that  inequality \eqref{eq17} is valid for any \(r_1\in  (0, 1)\)  and \(r_2>1\), since \( r_1 =1-\dfrac{1}{\nu}\) and \(r_2 =1+\dfrac{1}{\nu}\), it holds for  any \(\nu>2\). Hence, all conditions in Theorem \ref{tr4} are fulfilled with the Lyapunov function $V(u)=\dfrac{1}{2}\|u-u^*\|^2$ and with $c=\frac{1}{2}$. Using Theorem \ref{tr4} we derive that the conclusion of the theorem holds. 
\end{proof}

\begin{remark}
    Observe that an equilibrium point $u^*\in \R^n$ of the dynamical system \eqref{newydynamicalsystem} is global fixed-time stable, meaning that every trajectory of the system \eqref{newydynamicalsystem} converges strongly to $u^*$. Hence, instead of using a constant time step size as in \eqref{eq29} we discretize the time of the dynamical system \eqref{newydynamicalsystem} with varying step sizes, this means that the time step size, $\lambda$, depends on the number of iterations, $n$, i.e., $\lambda=\lambda(n)=\lambda_n$, then the iterated sequence \eqref{eq29} still converges to $u^*$. In this case, equation \eqref{eq29} can be rewritten as follows. 
    \begin{equation}\label{eq30} 
		u_{n+1}=u_n-\lambda_n \psi(u_n)(f(u_n)-P_{\Phi(u_n)}(f(u_n)-\alpha u_n)).
	\end{equation}
\end{remark}


\section{Numerical Illustration}\label{num}

In this section, we present numerical examples to illustrate the theoretical results established in the previous sections. Additionally, we apply our proposed method to the traffic assignment problem. To further demonstrate its effectiveness and advantages, we compare our algorithm with the approach presented in \cite{VuongThanh}. All implementations are carried out in MATLAB.

\begin{example} 
		
We consider a simple example in $\mathbb{R}^2$. Let $f:\mathbb{R}^2 \rightarrow \mathbb {R}^2$ defined by $ f(u)=A u$ where $$ A= \begin{bmatrix}  \ \ 3.2 & 2 \\ -0.6& 1\end{bmatrix}. $$ The IQVIP \eqref{iqvip} is considered  with $\Phi(u_1,u_2)= R(u_1,u_2)$ where $R(u_1,u_2)$ is the closed rectangle restricted by four lines $u= |u_1|, u=-|u_1|, v= |u_2|,v = -|u_2|$.\\
			
A straightforward verification shows that the condition \eqref{dk phi} is satisfied with $\mu=1$, and the matrix $A$ is positive definite with eigenvalues are 2.2 and 2. Consequently, $f$ is 2.2-Lipschitz continuous and 2-strongly monotone. This means $L=2.2$ and $\beta=2$. We choose $\alpha=2$. Then, condition \eqref{dk cho beta 2} is fulfilled. 
In addition, it is easy to check that $(0,0)$ is a solution of IQVIP \eqref{iqvip} and by Theorem \ref{unique sol},  IQVIP \eqref{iqvip} has a unique solution. Therefore, $u^*= (0,0)$ is a unique solution of the IQVIP \eqref{iqvip}. \\
	
In this experiment, we choose a quite small time step, $\lambda = 0.00146$, and other parameters $a_1=20; a_2=20; r_1=0.95; r_2=1.5$. With the initial point $(1,1)$, Figure \ref{Exam1} illustrates the convergence rates of the sequences generated by \eqref{eq29} ({\bf red line}) and by the projection method ({\bf blue line}) in \cite{VuongThanh} to the solution $u^*$. We observe that in this example, the sequence $(u_n)$	generated by our proposed algorithm converges to the solution $u^*$ much faster than the one in \cite{VuongThanh}.  Specifically, with only 100 steps, we have achieved an error $\|u_n-u^*\| < 10^{-4}$, whereas the sequence generated by the projection algorithm in \cite{VuongThanh} is still quite far from the solution $u^*$.\\
	
\begin{figure}[ht] 
    \centering
    \includegraphics[scale=0.9]{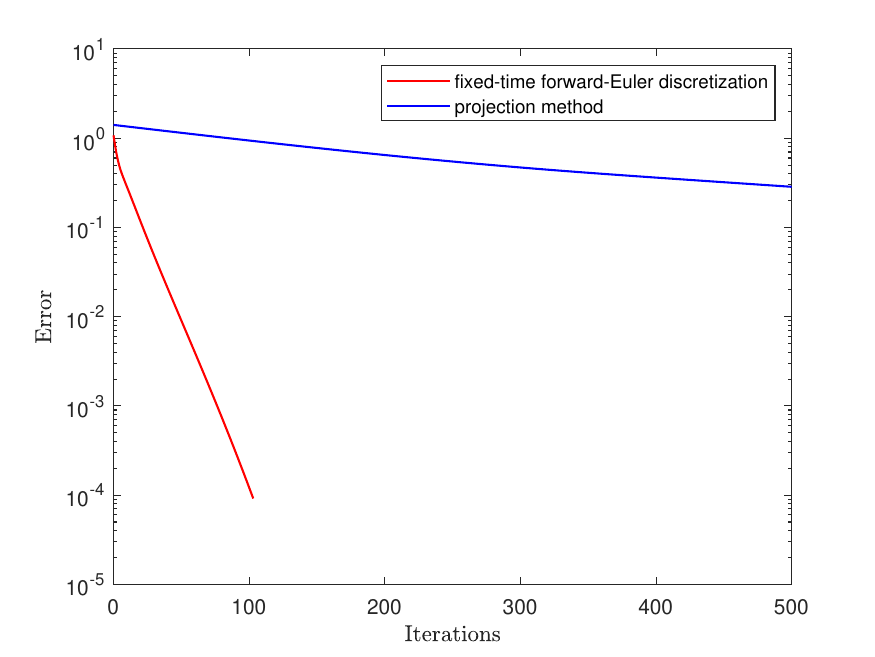}\\ 
    \caption{Performance of sequences generated by \eqref{eq29} and the projection method for same parameters.} 
    \label{Exam1}
\end{figure}
\end{example}

\begin{example}
{\bf Traffic Assignment Problems}\\
	We present a practical example adopted from \cite{VuongThanh}, which is referred to as a road pricing problem in traffic assignment. 
	In particular, we consider a continuous-time road pricing problem in which the authority seeks to manage the link flows, $f_j$, in the network by charging the link tolls, $u_j$, on a set of links $j \in J$ such that $f_j(u)$ belongs to the range of link flow $j$. We assume that the range for link flows depends on the fees imposed on the links, meaning the range of the link flows is a function of $u$, the fee on the links, i.e.,  $$\Phi(u) = \{u: a(u) \leq u\leq b(u)\}.$$
    
    
In \cite{VuongThanh}  this problem was established as an IQVIP: Find $u^*\in \R^n$ such that 
	\begin{equation*}
		g(u^*) \in - \Phi (u^*)\quad \text{and}    \quad \langle z + g(u^*), u^* \rangle \leq 0, \quad \forall z \in \Phi(u^*),
	\end{equation*}
	or equivalently
	\begin{equation}\label{iqvip exam2}
		g(u^*) \in - \Phi (u^*)\quad \text{and}    \quad \langle z - g(u^*), u^*\rangle  \geq 0, \quad \forall z \in -\Phi(u^*),
	\end{equation}
	with 	 $g=-f$. 
    
 Note that  $P_{-\Phi(u_n)}(g(u_n)-\alpha u_n) = -P_{\Phi(u_n)}(f(u_n) + \alpha u_n)$, and therefore the algorithm \eqref{eq30} for the IQVI problem \eqref{iqvip exam2} becomes
	\begin{align*} 
		u_{n+1} &= u_n + \lambda_n \psi(u_n) [P_{-\Phi(u_n)}(g(u_n)-\alpha u_n)- g(u_n) ]\\
		&= u_n + \lambda_n \psi(u_n)[ f(u_n) - P_{\Phi(u_n)}(f(u_n) + \alpha u_n)].
	\end{align*}
	
	We denote $R_n=\|\lambda_n \psi(u_n)(f(u_n) - P_{\Phi(u_n)}(f(u_n)+\alpha u_n))\|$ by the residual of Algorithm \eqref{eq29} and use it to illustrate the convergence rate of this algorithm.
	
	\begin{figure}[ht]
		\centering
		\caption{Road pricing problem with four bridge network} 
		\label{ODpic}
	\end{figure}

    As for the numerical experience, we reconsider the traffic network from \cite{VuongThanh} as shown in Figure \ref{ODpic}. This network consists of eight nodes and sixteen links connecting these nodes. We also reused the data from \cite{VuongThanh} presented in Tables \ref{ODdemand} and \ref{tf}.  As we can see, links 1, 2, 3, and 4 are four bridges over a river, connecting four origins $O_j$ with four	destinations $D_j$ and the demands between OD pairs given in Table \ref{ODdemand}. The goal is to ensure that the link flows remain within the specified range, i.e.,  $a(u) \leq f(u) \leq b(u)$, where $a(u)= u + A $ represents the lower bound and $b(u)=u+B$  represents the upper bound of the link flows. In this example, we still choose $A= (40,0,100)^T$ and $B= (90,50,200)^T$. We verify easily that the $\Phi(u)$ satisfies the condition \eqref{dk phi}.			
	\begin{table}[ht]
		\centering
		\caption{Origin-destination demand table} 
		\label{ODdemand}
	\end{table}
	
	\begin{table}[ht]
		\centering
		\caption{Link free flow travel time and capacity} 
		\label{tf}
	\end{table}
	Observe that link flows, $f_j(u)$, depend implicitly on the link tolls, $u ={u_j}$. As in \cite{VuongThanh}, the link flows $f_j$ are obtained by solving a fixed-demand user equilibrium traffic assignment, where the free flow travel times and the link capacities are given in Table \ref{tf} (the free flow travel times and the link capacities in link 5-16 are the same for each term). The link travel time $t_j$ on link $j$ follows the Bureau of Public Roads (BPR) function, i.e., 
	\[t_j(f_j)= t_j^0 \left[1+0.15 \left(\frac{f_j}{c_j}\right)^4\right],\]
	where $f_j,t_j^0$ and $c_j$ denote link flow, free flow travel time, and capacity on link $j$, respectively. Once again, we use the cost function from  \cite{VuongThanh}, which is the sum of travel time and imposed toll on each link. 
    
        \begin{figure}[ht]
		\centering
		\includegraphics[scale=0.35]{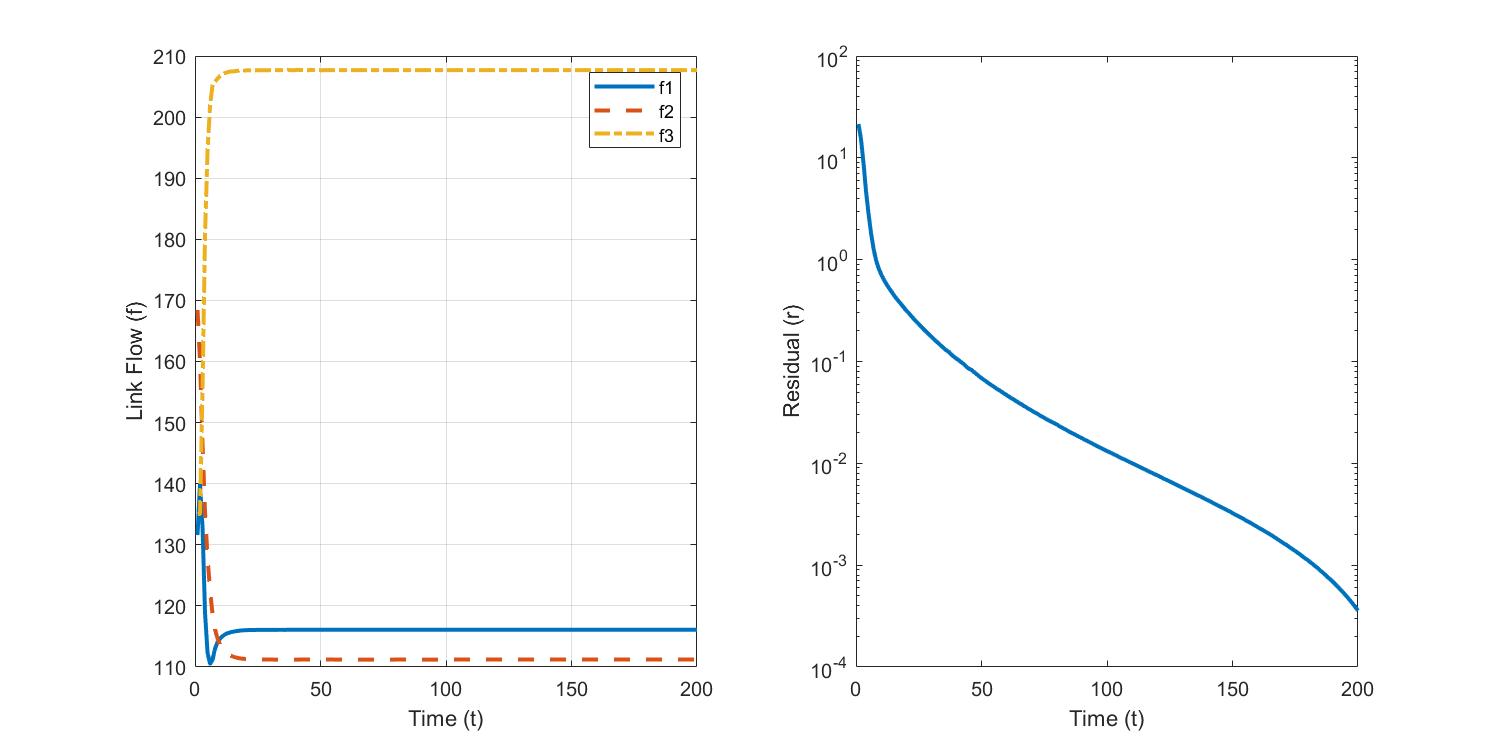}\\ 
		\caption{Flows on three bridges and convergence rate of the algorithm \ref{eq29}} 
		\label{Exam2}
	\end{figure}
    	
In our code, we fix scaling factor $\alpha =0.5$ and the parameters $a_1 = 0.75, a_2 = 0.75, r_1 = 0.65, r_2 = 1.5$. We choose stepsize of the form $\lambda_n = \frac{4}{n}$ to examine the stability of the dynamical system and the convergence of the iterated sequence \eqref{eq30}. Figure \ref{Exam2} shows link flows on three bridges and the convergence rate of the sequence \eqref{eq30}. We notice that after about 40 steps, the link flow on each bridge becomes stable at the level of 116.10, 111.22, 207.68 with the residual lower than 0.1. With the corresponding tolls on the three bridges being 26.09, 61.22, 7.69, the obtained link flows satisfy the given condition of the government.  
\end{example}

\section{Conclusion}\label{sec6}

In this paper, we present two modified projection dynamical systems for addressing inverse quasi-variational inequality problems—one ensuring finite-time stability and the other fixed-time stability. The latter guarantees the existence, uniqueness, and convergence of its trajectory to the unique solution of the inverse quasi-variational inequality problem within a fixed-time framework. Furthermore, it is shown that the fixed-time stability of the second dynamical system is preserved under the forward-Euler discretization. Numerical experiments are provided to illustrate the algorithm's performance. Future research may focus on investigating the finite-time and fixed-time stability of the modified projection dynamical systems in a broader setting, such as infinite-dimensional Hilbert or Banach spaces, with the potential to relax assumptions on operators and parameters.


\section*{Declarations} 
{\bf Conflict of interest} The authors declare no competing interests.

\end{document}